\documentclass[reqno]{amsart}
\usepackage{amsmath,amsfonts,amsthm,amssymb,stmaryrd,color}
\usepackage[normalem]{ulem} 
\usepackage[final]{showkeys} 
\usepackage[colorlinks=true,allcolors=blue]{hyperref}

\newcommand{\anorm}[1]{\| \widehat{#1}\|}
\renewcommand{\leq}{\leqslant}
\renewcommand{\geq}{\geqslant}
\newcommand{\ri}{\mathcal{R}}
\newcommand{\ind}[1]{1_{#1}}

\newcommand{\bbn}{\mathbf{N}}
\newcommand{\bbz}{\mathbf{Z}}
\newcommand{\bbr}{\mathbf{R}}

\newcommand{\bbc}{\mathbf{C}}

\newcommand{\supp}{\mathrm{supp}}
\newcommand{\td}{\,\mathrm{d}}
\newcommand*{\bbe}{
  \mathop{
    \mathchoice{\vcenter{\hbox{\larger[4]$\mathbb{E}$}}}
               {\kern0pt\mathbb{E}}
               {\kern0pt\mathbb{E}}
               {\kern0pt\mathbb{E}}
  }\displaylimits
}

\newcommand{\abs}[1]{\left\lvert #1\right\rvert}
\newcommand{\Abs}[1]{\lvert #1\rvert}
\newcommand{\brac}[1]{\left( #1\right)}

\newcommand{\norm}[1]{\left\lVert #1\right\rVert}

\numberwithin{equation}{section}
\newtheorem{theorem}{Theorem}[section]
\newtheorem{proposition}[theorem]{Proposition}
\newtheorem{lemma}[theorem]{Lemma}

\newtheorem{corollary}[theorem]{Corollary}
\newtheorem{conjecture}[theorem]{Conjecture}
\newtheorem*{question*}{Question}

\theoremstyle{definition}

\newtheorem{question}[theorem]{Question}
\newtheorem*{definition*}{Definition}

\theoremstyle{remark}
\newtheorem*{remark}{Remark}

\begin{document}
\title{Remarks on the inverse Littlewood conjecture}
\author{Thomas F. Bloom}
\address{Department of Mathematics, University of Manchester, Manchester, M13 9PL, UK}
\email{thomas.bloom@manchester.ac.uk}
\author{Ben Green}
\address{Mathematical Institute, Andrew Wiles Building, Radcliffe Observatory Quarter, Woodstock
Rd, Oxford OX2 6QW, UK}
\email{ben.green@maths.ox.ac.uk}
\begin{abstract}
The Littlewood conjecture, proven by Konyagin and McGehee-Pigno-Smith in the 1980s, states that if $A\subset \mathbb{Z}$ is a finite set of integers with $\lvert A\rvert=N$ then $\| \widehat{1_A}\|_1\geq  c\log N$ for some absolute constant $c > 0$. We explore what structure $A$ must have if $\| \widehat{1_A}\|_1\leq K\log N$ for some constant $K$. Under such an assumption we prove, for instance, that $A$ contains a subset $A'\subseteq A$ with $\lvert A'\rvert \geq N^{0.99}$ such that $|A' + A'| \ll K^{O(1)}|A'|$. As a consequence, for any $k\geq 3$, if $N$ is sufficiently large depending on $k$ and $K$, then $A$ must contain an arithmetic progression of length $k$. A byproduct of our analysis is a (slightly) improved bound for the constant $c$.

\end{abstract}
\maketitle

\section{Introduction}
For a finite set $A\subset\bbz$ the Fourier transform $\widehat{\ind{A}}: \bbr/\bbz \to\bbc$ is defined by $\widehat{\ind{A}}(\theta)=\sum_{n\in A}e(-n\theta)$, where $e(x) :=e^{2\pi ix}$. Define
\[\|\widehat{\ind{A}}\|_1 = \int_0^1 \lvert\widehat{\ind{A}}(\theta)\rvert\td \theta.\]
A famous conjecture of Littlewood stated that\footnote{We use the Vinogradov notation $f\gg g$, equivalent to $g=O(f)$, to denote the existence of some absolute constant $c>0$ such that $\abs{f(x)}\geq c\abs{g(x)}$ for all sufficiently large $x$.} $\anorm{\ind{A}}_1 \gg \log N$ for any set $A$ of $N$ integers, which is achieved for example by an arithmetic progression of length $N$. This was proved in 1981 independently by Konyagin \cite{Ko81} and McGehee, Pigno, and Smith \cite{MPS81}. A discussion of the problem and its history may be found in \cite[Chapter 10]{Choiment-Queffelec}. 

This lower bound is universal, yet rarely achieved -- in fact generically we expect $\anorm{\ind{A}}_1$ to be closer to the trivial upper bound $N^{1/2}$, which is an immediate consequence of the Cauchy-Schwarz inequality and Parseval's identity. This is achieved when $A$ is lacunary, for example, or a random positive density subset of an interval (with high probability). In fact, the only examples we know that achieve $O(\log N)$ are arithmetic progressions and simple variations, for example the union of $O(1)$ arithmetic progressions and some small unstructured set.

It is therefore natural, especially from the viewpoint of modern additive combinatorics, to ask the following inverse question.
\begin{question}
Let $K>0$ be some large constant. What can we say about the structure of sets $A\subset \bbz$ with $\abs{A}=N$ and $\anorm{\ind{A}}_1\leq K\log N$?
\end{question}
This question was asked by the second author \cite{Gr14} and one possible precise conjecture in this direction was formulated by Petridis \cite[Question 5.1]{Pe13} (see Section~\ref{sec-conj} for further details). While we are some way from a full resolution of this question, in this paper we offer the following weak partial progress.

Recall that the \emph{additive energy} $E(B)$ of a finite set $B \subset \bbz$ is defined by
\begin{equation}\label{add-energy-def} E(B) = \# \{ (b_1, b_2, b_3, b_4) \in B^4 : b_1 + b_2 = b_3 + b_4\}.\end{equation}
We have
\begin{equation}\label{energy-fourier} E(B) = \int^1_0 |\widehat{1_B}(\theta)|^4 d\theta,\end{equation} as can be checked using the orthogonality relation for characters. We also write $\omega[B] := E(B)/|B|^3$ for the normalised additive energy of $B$, which satisfies $0 < \omega[B] \leq 1$.

\begin{theorem}\label{th-struc}
Let $N$ be a sufficiently large positive integer. Let $\delta \in (0, \frac{1}{2}]$ and $K> 0$. If $A$ is a set of $N$ integers such that $\anorm{\ind{A}}_1\leq K\log N$ then there is a subset $A'\subseteq A$ of size $\abs{A'}\gg N^{1-\delta}$ such that $\omega[A']\gg (\delta/K)^{2}$.
\end{theorem}
The set $A'$ produced in Theorem~\ref{th-struc} is in fact an initial segment of $A$ .

The following two corollaries are almost immediate consequences of this and standard results in additive combinatorics.

\begin{corollary} \label{cor1}Let $N$ be a sufficiently large positive integer. Let $K > 0$ and suppose that $A\subset \bbz$ is a finite set of size $N$ such that $\anorm{\ind{A}}_1\leq K\log N$. Then there exists an arithmetic progression $P$ of length $\geq N^{c_K}$ such that $\abs{A\cap P}\geq c_K \abs{P}$, where $c_K > 0$ depends only on $K$.
\end{corollary}
\begin{proof} Apply Theorem \ref{th-struc} with $\delta = \frac{1}{2}$. Applying the Balog-Szemer\'{e}di-Gowers theorem (see, for example, \cite[Theorem 2.31]{TaVu06}) to the resulting set $A'$, we obtain a set $A'' \subseteq A$, $|A''| \gg K^{-O(1)} N^{1/2}$, with  $\abs{A''+A''}\leq K^{O(1)} \lvert A''\rvert$. By the Freiman-Ruzsa theorem (see, for example \cite[Theorem 5.33]{TaVu06}), there are arithmetic progressions $P_1,\ldots,P_r$ with $r = O_K(1)$ such that $A'' \subseteq P_1+\cdots+P_r$ and
\[\abs{P_1+\cdots+P_r}=\abs{P_1}\cdots\abs{P_r}\ll_K \abs{A''}.\]
By averaging there exists some $i$, $1\leq i\leq r$, and $x$ such that 
\[\abs{(P_i+x)\cap A''}\gg_K \abs{P_i}\gg_K \abs{A''}^{1/r}.\]
Since $N$ is large, the length of $P_i$ is bounded below by $N^{c_K}$ for some $c_K > 0$.
\end{proof}

\begin{corollary}\label{cor2}
Let $K>0$ and let $k \geq 3$ be an integer. Let $N$ be sufficiently large in terms of $k,K$ and let $A\subset \bbz$ be a finite set of size $N$ such that $\anorm{\ind{A}}_1\leq K\log N$. Then $A$ contains an arithmetic progression of length $k$.
\end{corollary}
\begin{proof}
This follows immediately from Corollary \ref{cor1} and Szemer\'edi's theorem.     
\end{proof}

\begin{remark}
The key feature of Theorem \ref{th-struc} is that the lower bound on $E(A')$ is a constant times $|A'|^3$. If one is prepared to tolerate logarithmic losses then a one line application of H\"older's inequality on the Fourier side shows that we in fact have $E(A) \geq N^3/(K\log N)^2$. However, this would certainly be too weak to deduce Corollary \ref{cor2} given our current knowledge of bounds in Szemer\'edi's theorem when $k \geq 4$.
\end{remark}

\subsection{Previous structural results}
We now summarise what was previously known concerning this inverse question for sets with small $\Vert \widehat{1_A} \Vert_1$. 
\begin{itemize}
\item As noted above, it is a trivial consequence of H\"{o}lder's inequality that $A$ itself must have reasonably large additive energy: $E(A)\geq N^3/(K\log N)^2$. 
\item Pichorides \cite[Lemma 3]{Pi80} noted that an argument of Zygmund \cite[Chapter XII, 7.6]{Zy02} implies that the largest dissociated\footnote{A set $X\subset \bbz$ is dissociated if all $2^{\abs{X}}$ subset sums are distinct.} subset of $A$ has size $O_K((\log N)^3)$. This was independently rediscovered by Bedert \cite{Be25}. 
\item Petridis \cite{Pe13} showed that $A$ cannot have genuine 3-dimensional structure, in a certain precise sense.
\item Picking up on the theme of Petridis, Hanson \cite{Ha21} showed that $A$ cannot have genuine 2-dimensional structure, again in a sense he was able to make precise.
\item If we replace the upper bound $\leq K\log N$ with $\leq K$, and replace $\bbz$ with a finite abelian group $G$, then a much more satisfactory answer is known: the second author and Sanders \cite{GrSa08} have shown that if $A\subseteq G$ has $\anorm{\ind{A}}_1\leq K$ then $\ind{A}$ can be written as a simple linear combination of $O_K(1)$ indicator functions of cosets of subgroups of $G$. This condition is (qualitatively) necessary and sufficient.
\end{itemize}
Part of the motivation for inverse results of this nature originated in work of Bourgain \cite{Bo97}, which highlighted a connection between Littlewood's conjecture and the task of improving the lower bounds for the sum-free subset problem in additive combinatorics. Bedert \cite{Be25} has made significant progress on this problem recently, making use of that connection.

\subsection{Bounds for the Littlewood conjecture}
The lower bound $\anorm{1_A}_1\gg \log N$ is the best possible up to an absolute constant, since this is achieved when $A$ is an arithmetic progression of length $N$. In this case, one may compute (see, for example, \cite[II.12.1]{Zy02}) that
\[\anorm{1_A}_1=\frac{4}{\pi^2}\log N+O(1).\]
One may ask (and Hardy and Littlewood did ask \cite[p 167]{HaLi48}) whether this leading constant is the best possible or even whether $\Vert \widehat{1_A} \Vert_1 \geq \Vert \widehat{1_P} \Vert_1$ where $P$ is an arithmetic progression of length $N$. This statement, or the slightly weaker statement that $\Vert \widehat{1_A} \Vert_1 \geq (\frac{4}{\pi^2} - o(1)) \log N$, is usually called the \emph{Strong Littlewood Conjecture}.

Less ambitiously, one can seek to prove the best possible value of $c$ in an inequality
\[\anorm{1_A}_1\geq (c-o(1))\log N.\]
The argument of McGehee, Pigno, and Smith \cite{MPS81} achieved $c=1/30$. The numerics were optimised by Stegeman \cite{St82} and Yabuta \cite{Ya82}, the latter proving a value of $c \approx 0.1295$. We give a slight improvement on this constant.

\begin{theorem}\label{th-const}
If $A\subset \bbz$ is a finite set of size $N$ then
\[\anorm{1_A}_1 \geq (c-o(1))\log N\]
where $c= 0.170934\cdots$.
\end{theorem}
For comparison note that $4/\pi^2\approx 0.405$ so we are still about a factor of $2.4$ away from the conjectured optimal value. The value of $c$ is found as the maximum of a complicated two-variable expression and presumably has no nicer form.

The proofs of Theorems~\ref{th-struc} and \ref{th-const} both hinge on a detailed analysis of the test functions used by McGehee, Pigno, and Smith. This is the main business of the paper.

In the final section of the paper we make some additional comments and record some further questions and conjectures.

\subsection*{Notation} For any function $f\in \ell^1(\bbz)$ we define the Fourier transform $\widehat{f}:[0,1]\to \bbc$ by
\[\widehat{f}(\theta) = \sum_{n} f(n) e(-n\theta),\]
where $e(x) = e^{2\pi ix}$. The convolution is defined in the usual way, so that if $f,g\in \ell^1(\bbz)$ then 
\[f\ast g(x) = \sum_{y} f(y)g(x-y).\]
Note that if $A=\supp(f)$ and $B=\supp(g)$ then $\supp(f\ast g)\subseteq A+B$. We note here the useful property that $\widehat{f\ast g}=\widehat{f}\cdot \widehat{g}$. It is convenient to use inner product notation, so that if $f,g:\bbz\to\bbc$ we write 
\[\langle f,g\rangle = \sum_{n\in \bbz}f(n)\overline{g(n)},\]
and $\langle \widehat{f},\widehat{g}\rangle = \int_0^1 \widehat{f}(\theta)\overline{\widehat{g}(\theta)}\td \theta$. By Parseval's identity
\[\langle f,g\rangle = \langle \widehat{f},\widehat{g}\rangle.\]
\subsection*{Acknowledgements} TB is supported by a Royal Society University Research Fellowship. BG is supported by a Simons Investigator award. We thank Benjamin Bedert for many insightful and productive discussions concerning the inverse Littlewood problem, and the referee for observing that a minor modification to the argument gives a larger value of $c$ than we had originally claimed.

\section{McGehee-Pigno-Smith test functions}
The overall strategy to prove Theorem~\ref{th-struc} is to mimic proofs of the Littlewood conjecture, aiming to prove a lower bound of the shape $\anorm{1_A}_1>K\log N$. Since this lower bound is false by assumption, some stage in our purported proof must fail, which will reveal structural information about $A$.

All proofs of the Littlewood conjecture follow a similar strategy, obtaining a lower bound for $\Vert \widehat{1}_A \Vert_1$ via construction of a suitable test function. In particular, our goal will be to try to construct a test function $\ri:\bbz\to\bbc$ such that
\begin{enumerate}
\item $\|\widehat{\ri}\|_\infty \leq 1$ and
\item $\langle \ind{A},\ri\rangle > K\log N$.
\end{enumerate}
By Parseval's identity and the triangle inequality these facts combined imply $\anorm{\ind{A}}_1>K\log N$. This contradicts our hypothesis, so the construction of such a test function must fail. The new observation is that, examining the details of the construction of the test function, this yields some additive information about $A$.

We will use the test functions of McGehee, Pigno, and Smith \cite{MPS81}. These are not the only ones capable of proving the Littlewood conjecture but, of the ones we are aware of, they perform the best from the quantitative point of view. For a discussion of other possible test functions we refer to the survey paper of Fournier \cite{Fo83}. 

The following lemma details the construction of these test functions.

\begin{lemma}\label{lem-MPS}
Let $b>0$ and set $c=1-e^{-b}$. If $f:\bbz \to \bbr$ is a finitely supported function then there exists a function $\ri_f:\bbz\to \bbc$ \textup{(}which depends on $b$\textup{)} with the following properties:
\begin{enumerate}
\item $\ri_f$ is supported on $\bbz_{\leq 0}$;
\item $\norm{\ri_f}_2\leq 2^{1/2}b\norm{f}_2$;
\item $\|\widehat{\ri_f}+1\|_\infty\leq 1$; and 
\item if $\|\widehat{f}\|_\infty\leq 1$ then, for any $g:\bbz\to \bbc$ with $\| \widehat{g}\|_\infty\leq 1$, the function
\[h :=g+cf+g\ast \ri_f\]
satisfies $\| \widehat{h}\|_\infty \leq 1$.
\end{enumerate}
\end{lemma}
\begin{proof}
Since $\Abs{\widehat{f}}$ is a $1$-periodic symmetric function we can write 
\[\Abs{\widehat{f}(\theta)}=\sum_{n\in \bbz}c_ne(n\theta)\]
for some $c_n\in \bbr$ with $c_n=c_{-n}$, and define $h_f:\bbz \to \bbc$ by
\[h_f(n) = \begin{cases} c_0&\textrm{if }n=0\\
2c_n&\textrm{ if }n<0\textrm{, and}\\
0&\textrm{otherwise,}\end{cases}\]
so that $h_f$ is supported on $\bbz_{\leq 0}$. Furthermore, since
\[\widehat{h_f}(\theta)=c_0+2 \sum_{n<0}c_ne(n\theta),\]
for any $\theta\in \bbr/\bbz$ we have
\[\Re \widehat{h_f}(\theta) = c_0+\sum_{n<0}c_n(e(n\theta)+e(-n\theta))= \Abs{\widehat{f}(\theta)} \geq 0.\]
By Parseval's theorem,
\[\norm{h_f}_2^2=c_0^2+2\sum_{n\neq 0}c_n^2\leq 2\sum_n c_n^2=2\| f\|_2^2.\]
We may now define
\[\ri_f(n) = \sum_{j\geq 1}\frac{(-b)^j}{j!} h_f^{(j)}(n),\]
where $h_f^{(j)}=h_f\ast\cdots\ast h_f$ denotes the $j$-fold convolution. Equivalently, we may define  $\ri_f$ via its Fourier transform by
\[\widehat{\ri_f}(\theta)=\sum_{j\geq 1}\frac{(-b)^j}{j!}\widehat{h_f}(\theta)^j=e^{-b\widehat{h_f}(\theta)}-1.\]
Since $\supp(g_1\ast g_2)\subseteq \supp(g_1)+\supp(g_2)$ for any pair of functions $g_1, g_2 : \bbz \rightarrow \bbc$, it is clear that $\ri_f$ is supported on $\bbz_{\leq 0}$, which is item (1). 

The inequality $\abs{e^{-z}-1}\leq \abs{z}$, valid whenever $\Re z\geq 0$, implies 
\[\Abs{\widehat{\ri_f}(\theta)} \leq b\Abs{\widehat{h_f}(\theta)}\]
pointwise, and hence in particular
\[\norm{\ri_f}_2\leq b\norm{h_f}_2\leq 2^{1/2}b\norm{f}_2,\] which is item (2).
We also have
\[\Abs{\widehat{\ri_f}(\theta)+1}= e^{-b\Re \widehat{h_f}(\theta)}=e^{-b\Abs{\widehat{f}(\theta)}}\leq 1\]
pointwise, which is (3). 

Finally, suppose $\|\widehat{f} \|_{\infty},\| \widehat{g}\|_\infty \leq 1$ and consider
\[h=cf+g+g\ast \ri_f,\] where $c = 1 - e^{-b}$.
Taking the Fourier transform, we have the pointwise inequality
\[\Abs{\widehat{h}}\leq \Abs{\widehat{g}}\Abs{1+\widehat{\ri_f}}+c\Abs{\widehat{f}}\leq e^{-b\Abs{\widehat{f}}}+(1 - e^{-b})\Abs{\widehat{f}} \leq 1,\] which is (4). 
In the final step here we used the inequality $e^{-bx}+(1 - e^{-b})x\leq 1$, which is valid for all $x$ with $0\leq x\leq 1$. 
\end{proof}

The form of item (4) lends itself to an iterative argument, the details of which are as follows.

\begin{lemma}\label{lem-test}
Let $b>0$ and $c=1-e^{-b}$. Let $J \in \bbn$ and suppose that $g_1,\ldots,g_J:\bbz\to \bbr$ are finitely supported functions such that $\| \widehat{g_i}\|_\infty\leq 1$ for all $i = 1,\dots,J$. Then there exists a test function $\ri:\bbz\to \bbc$ such that $\| \widehat{\ri}\|_\infty \leq 1$ and
\begin{equation}\label{r-form}\ri = c\sum_{j=1}^J g_j+\sum_{j=1}^JD_j,\end{equation}
where the function $D_j$ has the shape
\begin{equation}\label{d-shape} D_j = c\sum_{1\leq i<j}g_i\ast F_{ij},\end{equation}
where
\begin{equation}\label{fij-1} \| F_{ij}\|_2\leq 2^{1/2}b\norm{g_j}_2,\end{equation} and $F_{ij}$ is supported on $\bbz_{\leq 0}$.
\end{lemma}
\begin{proof}
We iteratively define functions $\ri_i : \bbz \rightarrow \bbc$ by $\ri_1 :=cg_1$ and
\[\ri_{i+1} :=\ri_i+cg_{i+1}+D_{i+1}\] for $i \geq 1$, where
\[D_{i+1} :=\ri_i\ast \ri_{g_{i+1}}.\]
(Here $\ri_{g_{i+1}}$ is the test function given by Lemma~\ref{lem-MPS} with $f=g_{i+1}$.) We then set $\ri=\ri_J$. The fact that $\ri$ has the form \eqref{r-form} is a trivial induction. 

By the final property of Lemma~\ref{lem-MPS} and a trivial induction, we have $\|\widehat{\ri_i}\|_\infty\leq 1$ for all $i$. A downwards induction on $i$ using the identity
\[\widehat{\ri_{i}}=c\widehat{g_{i}}+\widehat{\ri_{i - 1}}(\widehat{\ri_{g_{i}}}+1)\]
yields, for any $i<j$,
\[\widehat{\ri_j}=c\sum_{i<k\leq j}\widehat{g_k}\prod_{k<\ell\leq j}(\widehat{\ri_{g_\ell}}+1)+\widehat{\ri_i}\prod_{i<\ell\leq j}(\widehat{\ri_{g_\ell}}+1),\] where the product over $k < \ell \leq j$ is interpreted as $1$ when $k = j$.
In particular, taking $i = 1$ we have
\[\widehat{\ri_{j}}=c\sum_{1\leq i\leq j}\widehat{g_i}\prod_{i< \ell\leq j}(\widehat{\ri_{g_\ell}}+1).\]
From the definition of $D_j$ it then follows that
\[\widehat{D_j}=\widehat{\ri_{j-1}}\widehat{\ri_{g_j}}=c\sum_{1\leq i<j}\widehat{g_i}\widehat{\ri_{g_j}}\prod_{i< \ell<j}(\widehat{\ri_{g_\ell}}+1).\]
Defining
\[\widehat{F_{ij}} :=\widehat{\ri_{g_j}}\prod_{i< \ell < j}(\widehat{\ri_{g_\ell}}+1),\] we see that \eqref{d-shape} holds.

The bound \eqref{fij-1} now follows using Lemma \ref{lem-MPS} (2) and (3) and Parseval's identity $\Vert \widehat{F_{ij}} \Vert_2 = \Vert F_{ij} \Vert_2$. To see that $F_{ij}$ is supported on $\bbz_{\leq 0}$, observe that $F_{ij}$ is a sum of convolutions of functions $\mathcal{R}_{g_\ell}$, and each such function is supported on $\bbz_{\leq 0}$ by Lemma \ref{lem-MPS} (1).
\end{proof}

\section{A lower bound for $\anorm{1_A}_1$}

We will now apply the McGehee-Pigno-Smith test function from the previous section to prove the following technical proposition. It will be the key input in the proof of our two main results, Theorems \ref{th-struc} and \ref{th-const}. An initial segment of $A$ is a non-empty subset of the shape $A\cap (-\infty,x]$ for some $x\in \mathbb{R}$.

\begin{proposition}\label{prop-litt}
Let $b > 0$ be a real parameter, and let $J \in \bbn$. Let $A\subset\bbz$ be a finite set and suppose that $A_1,\dots, A_J$ are initial segments of $A$ with $A_1 \subseteq \cdots \subseteq A_J$. Let $\lambda > 1$ be a parameter, and suppose that $\abs{A_{j+1}}\geq \lambda \abs{A_{j}}$ for $j = 1,\dots, J-1$. Then
\[\anorm{\ind{A}}_1\geq (1-e^{-b})\brac{J- \frac{b}{\sqrt{\lambda} - 1}  \sum_{i = 1}^J(\omega[A_i]+\abs{A_i}^{-1})^{1/2}},\]
where $\omega[ \cdot]$ denotes the normalised additive energy. 
\end{proposition}
\begin{proof}
For $j = 1,\dots, J$, let $\mu_j=\frac{1}{\abs{A_j}}\ind{A_j}$. Then $\norm{\mu_j}_1=1$ and $\| \widehat{\mu_j}\|_\infty \leq 1$. Let $\ri$ be the associated McGehee-Pigno-Smith test function given by Lemma~\ref{lem-test}, with parameter $b$, so that $\| \widehat{\ri}\|_\infty \leq 1$ and
\[\ri=c\sum_{j = 1}^J \mu_{j}+\sum_{j = 1}^J D_j,\]
where $c=1-e^{-b}$. Since $\langle \mu_j,\ind{A}\rangle=1$ for all $j$, we deduce that
\begin{equation}\label{eq30} \anorm{\ind{A}}_1\geq \abs{\langle \ri,1_A\rangle}\geq cJ - \Big| \sum_j\langle D_j,\ind{A}\rangle \Big|.\end{equation}
It therefore suffices to prove an upper bound on $\big| \sum_j\langle D_j,\ind{A}\rangle\big|$ of sufficient quality. By \eqref{d-shape} we have $D_j=c\sum_{i<j}\mu_i\ast F_{ij}$,
and so 
\begin{equation}\label{eq31} \sum_{j = 1}^J \langle D_j,\ind{A}\rangle=c\sum_{1 \leq i<j \leq J}\langle 1_{A},\mu_i\ast F_{ij}\rangle.\end{equation}
Since $F_{ij}$ is supported on $\bbz_{\leq 0}$, from the nesting property of the $A_i$ we have 
\[ \langle 1_{A},\mu_i\ast F_{ij}\rangle = \langle 1_{A_{i}},\mu_i\ast F_{ij}\rangle = \langle 1_{A_i} \ast \mu_i^{\circ}, F_{ij}\rangle,\] where $g^{\circ}(x) := \overline{g(-x)}$.
Moreover by \eqref{fij-1} we have 
\[\| F_{ij}\|_2\leq 2^{1/2}b\norm{\mu_j}_2=2^{1/2}b\abs{A_j}^{-1/2}.\]
By the Cauchy-Schwarz inequality, we thus have, recalling that $F_{ij}$ is supported on $\bbz_{\leq 0}$,
\begin{equation}\label{eq32} \big| \langle 1_{A},\mu_i\ast F_{ij}\rangle\big| = \big| \langle 1_{A_{i}} \ast \mu_i^{\circ} , F_{ij}\rangle\big|  \leq 2^{1/2} b |A_j|^{-1/2} \big\Vert (1_{A_{i}} \ast \mu_i^{\circ})1_{\bbz_{\leq 0}} \big\Vert_2.\end{equation}
However one may compute that 
\[ \Vert (1_{A_{i}} \ast \mu_i^{\circ})1_{\bbz_{\leq 0}}\Vert_2 = \abs{A_i}^{-1}\brac{\frac{E(A_i)+\abs{A_i}^2}{2}}^{1/2}.\]
Combining this with \eqref{eq31} and \eqref{eq32} yields
\[ \Big| \sum_{j = 1}^J \langle D_j, 1_A\rangle \Big| \leq c b \sum_{1 \leq i < j \leq J}\frac{(E(A_i)+\abs{A_i}^2)^{1/2}}{|A_i| |A_j|^{1/2}}.\]
Substituting into \eqref{eq30} gives
\[\anorm{1_A}_1\geq (1-e^{-b})\Big( J-b\sum_{i = 1}^J\frac{(E(A_i)+\abs{A_i}^2)^{1/2}}{\abs{A_i}}\sum_{i<j\leq J}\abs{A_j}^{-1/2}\Big).\]
By the assumption in the proposition we have $|A_{j}|^{-1/2} \leq \lambda^{(i - j)/2} |A_i|^{-1/2}$ for $j \geq i$. Summing the geometric series yields
\[ \sum_{j > i} |A_j|^{-1/2} \leq \frac{\abs{A_i}^{-1/2}}{\sqrt{\lambda} - 1},\] and the stated conclusion follows using the definition $\omega[A_i] = E(A_i)/|A_i|^3$.
\end{proof}

\section{Improved constants for Littlewood's conjecture}

In the next two sections we deduce our two main theorems from Proposition~\ref{prop-litt}, starting here with Theorem \ref{th-const}, the improved constant for Littlewood's conjecture. 

One may deduce Littlewood's conjecture (with some constant) from Proposition~\ref{prop-litt} by taking $A_i$ to be roughly the $\approx \lambda^{i}$ smallest elements of $A$ for all $i$ and using the trivial bound $\omega[A_i] \leq 1$, choosing $b>0$ and $1/\lambda$ suitably small absolute constants yields $\anorm{1_A}_1\gg J$, where $\lambda^{J}\asymp N$ and hence $J\gg_{\lambda} \log N$ as required.

To extract a decent constant from this is a delicate optimisation in both $b$ and $\lambda$, and this was the task considered by \cite{St82,Ya82}. To get the slightly better constant in Theorem \ref{th-const}, we will use a slightly less trivial bound for the $\omega[A_i]$, obtained using a rearrangement inequality.

\begin{lemma}\label{lem-energy}
Suppose that $A, B$ are finite sets of integers with cardinalities $n,m$ respectively, where $m \leq n$. Then
\[E(A,B) = \# \{ (a_1, b_1, a_2, b_2) \in A \times B \times A \times B : a_1 + b_1 = a_2 + b_2 \}\leq nm^2\big( 1 - \frac{m}{3n}\big) + \frac{m}{3}.\] In particular, $\omega[A] \leq \frac{2}{3} + o_{n \rightarrow \infty}(1)$.
\end{lemma}
\begin{proof}
This follows from a well-known rearrangement inequality (see Gabriel \cite[Theorem 3]{Gabriel} or \cite[Theorem 376]{HLP52}\footnote{It is important to note, when reading this reference, the convention established on the previous page: `We may agree that, when there is no indication to the contrary, sums involving several suffixes are extended over values of the suffixes whose sum vanishes'.}). A quick proof is as follows. Let the elements of $A$ be $a_1 < a_2 < \cdots < a_n$, and the elements of $B$ be $b_1 < \cdots < b_m$. Writing $r(x)$ for the number of pairs $(a,b) \in A \times B$ with $a + b = x$, we have
\begin{equation}\label{order-bds} r(a_i + b_j) \leq 1 + \min(n - i, j - 1) + \min(m - j, i - 1).\end{equation} Indeed, the constant $1$ counts the representation $a_i + b_j$ itself. The $\min(n - i, j - 1)$ term is an upper bound for representations $a_i + b_j = a_{i'} + b_{j'}$ with $i < i' \leq n$, which must have also have $1 \leq j' < j$; for each $i'$ (or $j'$) there is at most one such representation. The $\min(m - j, i - 1)$ term is an upper bound for representations $a_i + b_j = a_{i'} + b_{j'}$ with $1 \leq i' < i$ (and hence $j < j' \leq m$).
Now sum \eqref{order-bds} over $1 \leq i \leq n$ and $1 \leq j \leq m$. The left-hand side becomes $E(A, B)$, and (after a computation) the right-hand side becomes $nm^2\big( 1 - \frac{m}{3n}\big) + \frac{m}{3}$.
\end{proof}

We now give the deduction of Theorem~\ref{th-const} from Proposition \ref{prop-litt}.

\begin{proof}[Proof of Theorem~\ref{th-const}] Let $b>0$ and $\lambda \geq 2$ be some constants to be chosen later, and fix some $A\subset \bbz$ of size $N$, for some large $N$ (where the notion of `large' may depend on $\lambda$ and $b$). We now construct nested initial segments $A_1 \subseteq \cdots \subseteq A_J$ of $A$, for a certain parameter $J$. 

Let $A_1$ be the set consisting of the smallest $\lfloor \log N\rfloor$ elements of $A$. In general, if $A_1,\ldots,A_j$ have been constructed, let $A_{j+1}$ be the set of the smallest $\lceil \lambda\abs{A_j}\rceil$ elements of $A$ and continue. Otherwise, we set $J=j$ and halt. Furthermore we have $\abs{A_{j+1}}\geq \lambda\abs{A_{j}}$ for $j = 1,\dots, J-1$. Additionally, $\abs{A_{j+1}}\leq \lambda\abs{A_j}+1$ for $j = 1,\dots, J-1$, and hence by induction
\[ |A_{j}| \leq \lambda^r \big(|A_{j-r}| + \lambda^{-1} + \lambda^{-2} + \cdots + \lambda^{-r}\big) \] for any $r \geq 1$.
Taking $r = j - 1$ (and since $\lambda \geq 2$) we have 
\[ \abs{A_{j}}\leq \lambda^{j-1}(1 + \log N ).\]
By definition, the parameter $J$ was minimal such that
\[N< \lambda\abs{A_J}.\] It follows that 
\[ N < \lambda^{J}(1 + \log N)\]
and so
\begin{equation}\label{J-lower}  J\geq (1-o(1))\frac{\log N}{\log \lambda}.\end{equation}
Here, and below, $o(1)$ means a quantity tending to $0$ as $N \rightarrow \infty$.

Applying Lemma \ref{lem-energy} to $A_1,\dots, A_J$ and substituting in to Proposition~\ref{prop-litt}, and recalling our bound \eqref{J-lower} on $J$, we deduce that
\[\anorm{1_A}_1\geq (1-o(1))f(b,\lambda)\log N\]
where
\[f(b,\lambda)=\frac{1-e^{-b}}{\log \lambda}\brac{1-(2/3)^{1/2}\frac{b}{\sqrt{\lambda} - 1}}.\]
It remains to choose $b>0$ and $\lambda\in [2,\infty)$ to maximise $f(b,\lambda)$. Computational investigations tell us that the maximum value is
\[f(b,\lambda)=0.170934\cdots\]
achieved at
\[b=0.932199\cdots \quad \textrm{and}\quad \lambda=9.112038\cdots.\qedhere\]
\end{proof}

\section{A structural inverse result}

Now we turn to the deduction of Theorem \ref{th-struc}.
In deducing Theorem~\ref{th-struc} we are not concerned with constants, and so we record the following simpler consequence of Proposition~\ref{prop-litt}.

\begin{corollary}\label{cor-litt}
Let $\eta\in(0,1/4]$. Let $A\subset\bbz$ be a finite set. Suppose we have a nested sequence $A_1\subseteq A_2\subseteq \cdots \subseteq A_J$ of initial segments of $A$
such that
\[\abs{A_{i-1}}\leq (1-\eta)\abs{A_i}.\]
There exists a constant $c>0$ such that 
\[\anorm{\ind{A}}_1\gg J -c\eta^{-1} \sum_{i = 1}^J  \omega[A_i]^{1/2}.\]
\end{corollary}
\begin{proof}
Noting that $\omega[A] \geq |A|^{-1}$ for any nonempty set $A$ of integers, this follows immediately from Proposition~\ref{prop-litt} with $b=1$ and $1/\lambda=1-\eta$, so that $\frac{1}{\sqrt{\lambda} - 1} \ll \eta^{-1}$.
\end{proof}

We may now prove Theorem~\ref{th-struc}.
\begin{proof}[Proof of Theorem~\ref{th-struc}] We may assume that $K \geq \frac{1}{100}$ (say) since the Littlewood conjecture is true. If $\delta/K \leq N^{-1/2}$ then the result is trivial by taking $A' = A$, since any set of size $N$ has at least $N^2$ additive quadruples.  If $\delta \leq C(\log N)^{-1}$ for some absolute constant $C>0$ then the result is again true with $A' = A$ by applying H\"older's inequality on the Fourier side, noting that $E(A)^{1/4}$ is the $L^4$-norm of $\widehat{1}_A$.

In what follows we suppose that $\delta/K > N^{-1/2}$ and that $\delta \geq C (\log N)^{-1}$ (where $C>0$ is some absolute constant to be chosen later).
Set $\eta := c\delta/K$, where $c \in (0,\frac{1}{4}]$ is an absolute constant to be specified later. In particular, $\eta^{-1} \ll N^{1/2}$.

We construct nested initial segments $A_1 \subseteq A_2 \subseteq \cdots \subseteq A_J$ of $A$ in the following fashion. Let $A_1$ be the smallest $\lfloor N^{1-\delta}\rfloor$ elements of $A$. In general, if $A_1,\ldots,A_j$ have been constructed, then let $A_{j+1}$ be the set of the smallest $\lceil (1-\eta)^{-1}\abs{A_j}\rceil$ elements of $A$ if $A$ has this many elements, or else halt with $j = J$. By construction it is clear that $\abs{A_j}\leq (1-\eta)\abs{A_{j+1}}$ for all $j \in \{1,\dots, J-1\}$. Furthermore, $\abs{A_{j+1}}\leq (1-\eta)^{-1}\abs{A_j}+1$ for $j \in \{1,\dots, J-1\}$, and hence by the same induction leading to \eqref{J-lower} we have
\begin{equation}\label{aj11}\abs{A_{j}}\leq (1-\eta)^{1-j}(N^{1-\delta}+ \eta^{-1}) \ll (1 - \eta)^{1 - j} N^{1 - \delta}.\end{equation}

The parameter $J$ was the minimal $j\geq 1$ such that
\[N< (1-\eta)^{-1}\abs{A_J},\]
and hence by \eqref{aj11} we have $N^\delta \ll (1-\eta)^{-J} \leq e^{2\eta J}$. Since $\delta \geq C(\log N)^{-1}$ (and $N$ is sufficiently large), provided $C>0$ is chosen sufficiently large, it follows that $J\gg \eta^{-1}\delta \log N$. 

By Corollary~\ref{cor-litt} we see that either
\begin{equation}\label{former} \sum_{i = 1}^J \omega[A_i]^{1/2}\gg \eta J\end{equation}
or $\anorm{1_A}_1 \gg \eta^{-1}\delta \log N$. If the constant $c$ is chosen appropriately then, since $\eta = c \delta/K$, the latter option is contrary to our assumption. Thus \eqref{former} holds and so there exists some $i$ such that
\[ \omega[A_i] \gg \eta^2\gg (\delta/K)^2.\]
The conclusion now follows taking $A'=A_{i}$ (noting that since $A_1\subseteq A'$ we always have $\abs{A'}\geq \lfloor N^{1-\delta}\rfloor$). 
\end{proof}

\section{Speculations about stronger structure}\label{sec-conj}
The structure found by Theorem~\ref{th-struc} is, although non-trivial, still quite weak compared to the kind of structure one expects to find. For instance, we do not know of a counterexample to the following conjecture.
\begin{conjecture}\label{conj-1}
Let $A\subset \bbz$ be a finite set of size $N$. If $\anorm{1_A}_1\leq K\log N$ then for some $m,r\ll_K 1$ the following holds. There are finite arithmetic progressions $P_1,\ldots,P_m\subset \mathbb{Z}$, each of length at most $N$, and finite sets $X_1,\ldots,X_r\subset \bbz$ such that $\big| \bigcup_{j}X_j\big|=o_K(N)$,
together with $\epsilon_1,\ldots,\epsilon_m,\eta_1,\ldots,\eta_r\in \{-1,1\}$ such that
\[1_A=\sum_{1\leq i\leq m}\epsilon_i 1_{P_i}+\sum_{1\leq j\leq r}\eta_j 1_{X_j}.\]
\end{conjecture}
A very similar (but slightly stronger) conjecture was put forward in print by Petridis \cite[Question 5.1]{Pe13}. Analogous results are known to be true in finite field settings due to work of Sanders and the second author \cite{GrSa08}.

A consequence of Conjecture \ref{conj-1}, still much stronger than anything we can prove, would be that there exists an arithmetic progression $P$ of length at most $N$ such that $|A \cap P| \gg_K N$. 

A third possible inverse statement concerns the `dimension' of $A$, defined as the size of the largest dissociated subset of $A$. As noted in the introduction, Pichorides \cite{Pi80} has proved that if $\anorm{1_A}_1\leq K\log N$ then
\[\dim A \ll_K (\log N)^3.\]
In the other direction, there are examples where $\anorm{1_A}_1\ll \log N$ and yet $\dim(A) \gg (\log N)^2$ -- simply take the union of an arithmetic progression $P$ of size $\approx N$ and a dissociated set of size $\approx X = \lfloor (\log N)^2 \rfloor$, and use the triangle inequality together with the observation that
\[\anorm{1_X}_1\leq \abs{X}^{1/2}\ll \log N.\]
We conjecture that this is the best possible.
\begin{conjecture}
Let $K>0$ and $N$ be sufficiently large depending only on $K$. Let $A\subset \bbz$ be a finite set of size $N$. If $\anorm{1_A}_1\leq K\log N$ then $\dim A\ll_K (\log N)^2$. Furthermore, $A$ contains a subset of size $\gg_K N$ and dimension $\ll_K \log N$.
\end{conjecture}

\bibliographystyle{plain}
\bibliography{WeakInverse} 
\end{document}